\documentclass{amsart}
\usepackage{amsfonts}
\usepackage{amssymb,latexsym,amsmath}
\usepackage[draft=true]{hyperref}
\usepackage{cite}
\usepackage[left=3cm,right=3cm,top=2.5cm,bottom=2.5cm]{geometry}
\usepackage{hyperref}

\theoremstyle{plain}
\newtheorem{theorem}{Theorem}
\newtheorem{lemma}{Lemma}

\newtheorem{proposition}{Proposition}

\theoremstyle{definition}
\newtheorem{definition}{Definition}
\theoremstyle{remark}
\newtheorem{remark}{Remark}
\newtheorem{example}{Example}

\numberwithin{equation}{section}

\begin{document}
\title[Hermite-Hadamard type inequalities for the interval-valued
approximately\ldots ]{On the Hermite-Hadamard inequalities for
interval-valued co-ordinated convex functions}
\author{Dafang Zhao}
\address{College of Science, Hohai University, Nanjing, P. R. China and
School of Mathematics and Statistics, Hubei Normal University, Huangshi, P.
R. China}
\email{dafangzhao@163.com}
\author{Muhammad Aamir Ali}
\address{Jiangsu Key Laboratory for NSLSCS, School of Mathematical Sciences,
Nanjing Normal University, 210023, China}
\email{mahr.muhammad.aamir@gmail.com}
\author{Ghulam Murtaza}
\address{Department of Mathematics,Government College University Faisalabd,
Pakistan}
\email{gmnizami@gmail.com}
\keywords{\textbf{}Interval--valued functions, Co-ordinated convex
functions, Hermite--Hadamard inequalities.}

\begin{abstract}
In this paper, we establish Hermite-Hadamard inequality for interval-valued
convex function on the co-ordinates on the rectangle from the plane. We also
present Hermite-Hadamard inequality for the product of interval-valued
convex functions on co-ordinates. Some examples are also given to clarify
our new results.
\end{abstract}

\maketitle

\section{{\protect \large {Introduction}}}

\noindent The Hermite--Hadamard inequality discovered by C. Hermite and J.
Hadamard, (see \cite{Dragomir1}, \cite[pp. 137]{Pecaric}) is one of the most
well established inequalities in the theory of convex functions with a
geometrical interpretation and many applications. These inequalities state
that, if $f:I\rightarrow \mathbb{R}$ is a convex function on the interval $I$
of real numbers and $a,b\in I$ with $a<b$, then
\begin{equation}
f\left( \frac{a+b}{2}\right) \leq \frac{1}{b-a}\int
\limits_{a}^{b}f(x)dx\leq \frac{f\left( a\right) +f\left( b\right) }{2}.
\label{E1}
\end{equation}%
Both inequalities in (\ref{E1}) hold in the reversed direction if $f$ is
concave. We note that Hermite--Hadamard inequality may be regarded as a
refinement of the concept of convexity and it follows easily from Jensen's
inequality. Hermite--Hadamard inequality for convex functions has received
renewed attention in recent years and a remarkable variety of refinements
and generalizations have been studied.

In \cite{Dragomir}, Dragomir established the following similar inequality of
Hadamard type for the co-ordinated convex functions.

\begin{theorem}
\label{Dr1}Let $f:\Delta =[a,b]\times \lbrack c,d]\rightarrow
\mathbb{R}
$ is convex on co-ordinates $\Delta $. Then following inequalities holds:%
\begin{eqnarray}
f\left( \frac{a+b}{2},\frac{c+d}{2}\right)  &\leq &\frac{1}{2}\left[ \frac{1%
}{b-a}\int_{a}^{b}f\left( x,\frac{c+d}{2}\right) dx+\frac{1}{d-c}%
\int_{c}^{d}f\left( \frac{a+b}{2},y\right) dy\right]   \label{dr1} \\
&\leq &\frac{1}{(b-a)(d-c)}\int_{a}^{b}\int_{c}^{d}f(x,y)dydx  \notag \\
&\leq &\frac{1}{4}\left[ \frac{1}{b-a}\int_{a}^{b}f(x,c)dx+\frac{1}{b-a}%
\int_{a}^{b}f(x,d)dx\right.   \notag \\
&&\left. +\frac{1}{d-c}\int_{c}^{d}f(a,y)dy+\frac{1}{d-c}\int_{c}^{d}f(b,y)dy%
\right]   \notag \\
&\leq &\frac{f(a,c)+f(a,d)+f(b,c)+f(b,d)}{4}.  \notag
\end{eqnarray}
\end{theorem}

For more results related to (\ref{dr1}) we refer (\cite{Almori},\cite{AL},%
\cite{MZ}) and references therein.

On the other hand, interval analysis is a particular case of set--valued
analysis which is the study of sets in the spirit of mathematical analysis
and general topology. It was introduced as an attempt to handle interval
uncertainty that appears in many mathematical or computer models of some
deterministic real--world phenomena. An old example of interval enclosure is
Archimede's method which is related to compute of the circumference of a
circle. In 1966, the first book related to interval analysis was given by
Moore who is known as the first user of intervals in computational
mathematics, see \cite{Moore}. After his book, several scientists started to
investigate theory and application of interval arithmetic. Nowadays, because
of its applications, interval analysis is a useful tool in various area
which are interested intensely in uncertain data. You can see applications
in computer graphics, experimental and computational physics, error
analysis, robotics and many others.

What's more, several important inequalities (Hermite--Hadamard, Ostrowski,
etc.) have been studied for the interval--valued functions in recent years.
In \cite{CAno,Cano1}, Chalco--Cano et al. obtained Ostrowski type
inequalities for interval--valued functions by using Hukuhara derivative for
interval--valued functions. In \cite{flo2}, Rom\'{a}n-Flores et al.
established Minkowski and Beckenbach's inequalities for interval--valued
functions. For the others, please see \cite{flo2,costa,costa2,flo,flo3}.
However, inequalities were studied for more general set--valued maps. For
example, in \cite{Sado}, Sadowska gave the Hermite--Hadamard inequality. For
the other studies, you can see \cite{Mitroi,niko1}.\newline

\section{Preliminaries and Known Results}

In this section we recalling some basics definitions, results, notions and
properties, which are used throughout the paper. We denote $%
\mathbb{R}
_{\mathcal{I}}^{+}$ the family of all positive intervals of $%
\mathbb{R}
$. The Hausdorff distance between $[\underline{X},\overline{X}]$ and $[%
\underline{Y},\overline{Y}]$ is defined as%
\begin{equation*}
d([\underline{X},\overline{X}],[\underline{Y},\overline{Y}])=\max \left \{
\left \vert \underline{X}-\underline{Y}\right \vert ,\overline{X}-\overline{Y%
}\right \} .
\end{equation*}
The $(%
\mathbb{R}
_{\mathcal{I}},d)$ is a complete metric space. For more details and basic
notations on interval-valued functions see (\cite{Moore2}, \cite{Zhao}).

It is remarkable that Moore \cite{Moore} introduced the Riemann integral for
the interval-valued functions. The set of all Riemann integrable
interval-valued functions and real-valued functions on $[a,b]$ are denoted
by $\mathcal{IR}_{([a,b])}~$and $\mathcal{R}_{([a,b])},$ respectively. The
following theorem gives relation between $(IR)$--integrable and Riemann
integrable ($R$--integrable) (see \cite{Moore2}, pp. 131):

\begin{theorem}
\label{t1} Let $F:\left[ a,b\right] \rightarrow
\mathbb{R}
_{\mathcal{I}}$ be an interval--valued function such that $F(t)=\left[
\underline{F}(t),\overline{F}(t)\right] .$ $F\in \mathcal{IR}_{(\left[ a,b%
\right] )}$ if and only if \underline{$F$}$(t)$, $\overline{F}(t)\in
\mathcal{R}_{(\left[ a,b\right] )}$ and
\begin{equation*}
(IR)\int \limits_{a}^{b}F(t)dt=\left[ (R)\int \limits_{a}^{b}\underline{F}%
(t)dt,(R)\int \limits_{a}^{b}\overline{F}(t)dt\right] .
\end{equation*}
\end{theorem}

\noindent In \cite{Zhao,zhao2}, Zhao et al. introduced a kind of convex
interval--valued function as follows:

\begin{definition}
\label{d111} Let $h:[c,d]\rightarrow
\mathbb{R}
$ be a non--negative function, $(0,1)\subseteq \lbrack c,d]$ and $h\neq 0.$
We say that $F:[a,b]\rightarrow
\mathbb{R}
_{\mathcal{I}}^{+}$ is a $h$--convex interval--valued function, if for all $%
x,y\in \lbrack a,b]$ and $t\in (0,1),$ we have
\begin{equation}
h(t)F(x)+h(1-t)F(y)\subseteq F(tx+(1-t)y).  \label{i1}
\end{equation}%
With $SX(h,[a,b],%
\mathbb{R}
_{\mathcal{I}}^{+})$ will show the set of all $h$--convex interval--valued
functions.
\end{definition}

\noindent The usual notion of convex interval--valued function corresponds
to relation (\ref{i1}) with $h(t)=t,$ see \cite{Sado}. Also, if we take $%
h(t)=t^{s}$ in (\ref{i1}), then Definition \ref{d111} gives the other convex
interval--valued function defined by Breckner, see \cite{Breckner}.

\noindent Otherwise, Zhao et al. obtained the following Hermite--Hadamard
inequality for interval--valued functions by using $h$--convex:

\begin{theorem}
\label{zhtm1} \cite{Zhao} Let $F:[a,b]\rightarrow
\mathbb{R}
_{\mathcal{I}}^{+}$ be an interval--valued function such that $F(t)=[%
\underline{F}(t),\overline{F}(t)]$ and $F\in \mathcal{IR}_{(\left[ a,b\right]
)},h:[0,1]\rightarrow
\mathbb{R}
$ be a non--negative function and $h\left( \frac{1}{2}\right) \neq 0.$ If $%
F\in SX(h,[a,b],%
\mathbb{R}
_{\mathcal{I}}^{+})$, then
\begin{equation}
\frac{1}{2h\left( \frac{1}{2}\right) }F\left( \frac{a+b}{2}\right) \supseteq
\frac{1}{b-a}(IR)\int \limits_{a}^{b}F(x)dx\supseteq \lbrack F(a)+F(b)]\int
\limits_{0}^{1}h(t)dt.  \label{k1}
\end{equation}
\end{theorem}

\begin{remark}
\noindent (i) If $h(t)=t,$ then (\ref{k1}) reduces to the following result:
\begin{equation}
F\left( \frac{a+b}{2}\right) \supseteq \frac{1}{b-a}(IR)\int%
\limits_{a}^{b}F(x)dx\supseteq \frac{F(a)+F(b)}{2},  \label{k2}
\end{equation}%
\noindent which is obtained by \cite{Sado}.\newline

\noindent (ii) If $h(t)=t^{s},$ then (\ref{k1}) reduces to the following
result:
\begin{equation*}
2^{s-1}F\left( \frac{a+b}{2}\right) \supseteq \frac{1}{b-a}(IR)\int
\limits_{a}^{b}F(x)dx\supseteq \frac{F(a)+F(b)}{s+1},
\end{equation*}%
which is obtained by \cite{Gomez}.
\end{remark}

\begin{theorem}
\label{zhtm2}Let $F,G:[a,b]\rightarrow
\mathbb{R}
_{\mathcal{I}}^{+}$ be two interval-valued functions such that $F(t)=[%
\underline{F}(t),\overline{F}(t)]$ and $G(t)=[\underline{G}(t),\overline{G}%
(t)],$ where $F,G\in \mathcal{IR}_{(\left[ a,b\right] )},h_{1},h_{2}:[0,1]%
\rightarrow
\mathbb{R}
$ are two non--negative functions and $h_{1}\left( \frac{1}{2}\right)
h_{2}\left( \frac{1}{2}\right) \neq 0.$ If $F,G\in SX(h,[a,b],%
\mathbb{R}
_{\mathcal{I}}^{+})$, then
\begin{eqnarray}
&&\frac{1}{2h_{1}\left( \frac{1}{2}\right) h_{2}\left( \frac{1}{2}\right) }%
F\left( \frac{a+b}{2}\right) G\left( \frac{a+b}{2}\right)  \label{k3} \\
&\supseteq &\frac{1}{b-a}(IR)\int
\limits_{a}^{b}F(x)G(x)dx+M(a,b)\int_{0}^{1}h_{1}(t)h_{2}(1-t)dt  \notag \\
&&+N(a,b)\int_{0}^{1}h_{1}(t)h_{2}(t)dt  \notag
\end{eqnarray}%
\noindent and%
\begin{equation}
\frac{1}{b-a}(IR)\int_{a}^{b}F(x)G(x)dx\supseteq
M(a,b)\int_{0}^{1}h_{1}(t)h_{2}(t)dt+N(a,b)\int_{0}^{1}h_{1}(t)h_{2}(1-t)dt,
\label{k4}
\end{equation}%
where%
\begin{equation*}
M(a,b)=F(a)G(a)+F(b)G(b)\text{ and }N(a,b)=F(a)G(b)+F(b)G(a).
\end{equation*}
\end{theorem}

\begin{remark}
If $h(t)=t,$ the (\ref{k3}) reduces to the following result:
\begin{equation}
\frac{1}{b-a} (IR)\int_{a}^{b}F(x)G(x)dx\supseteq \frac{1}{3}M(a,b)+\frac{1}{%
6}N(a,b).  \label{k5}
\end{equation}
\end{remark}

\begin{remark}
If $h(t)=t,$ then (\ref{k4}) reduces to the following result:
\begin{equation}
2F\left( \frac{a+b}{2}\right) G\left( \frac{a+b}{2}\right) \supseteq \frac{1%
}{b-a}(IR)\int_{a}^{b}F(x)G(x)dx+\frac{1}{6}M(a,b)+\frac{1}{3}N(a,b).
\label{k6}
\end{equation}
\end{remark}

\section{Interval-valued double integral}

A set of numbers $\{t_{i-1},\xi _{i},t_{i}\}_{i=1}^{m}$ is called tagged
partition $P_{1}$ of $[a,b]$ if%
\begin{equation*}
P_{1}:a=t_{0}<t_{1}<\ldots <t_{n}=b
\end{equation*}%
and if $t_{i-1}\leq \xi _{i}\leq t_{i}$ for all $i=1,2,3,\ldots ,m$.
Moreover if we have $\Delta t_{i}=t_{i}-t_{i-1}$, then $P_{1}$ is said to be
$\delta -$fine if $\Delta t_{i}<\delta $ for all $i.$ Let $\mathcal{P}%
(\delta ,[a,b])$ denote the set of all $\delta -$fine partitions of $[a,b].$
If $\{t_{i-1},\xi _{i},t_{i}\}_{i=1}^{m}$ is a $\delta -$fine $P_{1}$ of $%
[a,b]$ and if $\{s_{j-1},\eta _{j},t_{j}\}_{j=1}^{n}$ is $\delta -$fine $%
P_{2}~$of $[c,d],$ then rectangles%
\begin{equation*}
\Delta _{i,j}=[t_{i-1},t_{i}]\times \lbrack s_{j-1},s_{j}]
\end{equation*}%
partition the rectangle $\Delta =[a,b]\times \lbrack c,d]$ and the points $%
(\xi _{i},\eta _{j})$ are inside the rectangles $[t_{i-1},t_{i}]\times
\lbrack s_{j-1},s_{j}].$ Further, by $\mathcal{P}$ ($\delta ,\Delta $) we
denote the set of all $\delta -$fine partitions $P$ of $\Delta $ with $%
P_{1}\times P_{2},$ where $P_{1}\in \mathcal{P}$($\delta ,[a,b]$) and $%
P_{2}\in \mathcal{P}(\delta ,[c,d]).$ Let $\Delta A_{i,~j}$ be the area of
rectangle $\Delta _{i,j}.$ In each rectangle $\Delta _{i,j}$, where 1$\leq
i\leq m,~1\leq j\leq n$, choose arbitrary $(\xi _{i},\eta _{j})$ and get%
\begin{equation*}
S(F,P,\delta ,\Delta )=\sum \limits_{i=1}^{m}\sum_{j=1}^{n}F(\xi _{i},\eta
_{j})\Delta A_{i,~j}.
\end{equation*}%
We call $S(F,P,\delta ,\Delta )$ is integral sum of $F$ associated with $%
P\in \mathcal{P}(\delta ,\Delta ).$

Now we recall the concept of interval-valued double integral given by Zhao
et al. in \cite{zhao3}.

\begin{theorem}
\cite{zhao3}Let $F:\Delta \rightarrow
\mathbb{R}
_{\mathcal{I}}.$ Then $F$ is called $ID-$integrable on $\Delta $ with $ID-$%
integral $U=(ID)\iint \limits_{\Delta }F(t,s)dA,$ if for any $\epsilon >0$
there exist $\delta >0$ such that
\begin{equation*}
d(S(F,P,\delta ,\Delta ))<\epsilon
\end{equation*}

for any $P\in \mathcal{P}$($\delta ,\Delta $). The collection of all $ID-$%
integrable functions on $\Delta $ will be denoted by $\mathcal{I}D_{(\Delta
)}.$
\end{theorem}

\begin{theorem}
\cite{zhao3}Let $\Delta =[a,b]\times \lbrack c,d]$. If $F:\Delta \rightarrow
\mathbb{R}
_{\mathcal{I}}$ is $ID-$integrable on $\Delta $, then we have%
\begin{equation*}
(ID)\iint \limits_{\Delta
}F(s,t)dA=(IR)\int_{a}^{b}(IR)\int_{c}^{d}F(s,t)dsdt.
\end{equation*}
\end{theorem}

\begin{example}
Let $F:\Delta =[0,1]\times \lbrack 1,2]\rightarrow
\mathbb{R}
_{\mathcal{I}}^{+}$ be defined by%
\begin{equation*}
F(s,t)=[st,s+t],
\end{equation*}%
then $F(s,t)$ is integrable on $\Delta $ and ($ID)\iint \limits_{\Delta
}F(t,s)dA=\left[ \frac{3}{4},2\right] .$
\end{example}

\section{\protect \Large Main Results}

In this section, we define interval-valued co-ordinated convex function and
prove some inequalities of Hermite-Hadamard type by using our new
definition. Throughout this section we will use $\Delta =[a,b]\times \lbrack
c,d]$, where $a<b$ and $c<d$, $a,b,c,d\in
\mathbb{R}
.$

\begin{definition}
A function $F:\Delta \rightarrow
\mathbb{R}
_{\mathcal{I}}^{+}$ is said to be interval-valued co-ordinated convex
function, if the following inequality holds:%
\begin{eqnarray*}
&&F(tx+(1-t)y,su+(1-s)w) \\
&\supseteq &tsF(x,u)+t(1-s)F(x,w)+s(1-t)F(y,u)+(1-s)(1-t)F(y,w),
\end{eqnarray*}

for all $(x,y),(u,w)\in \Delta $ and $s,t\in \lbrack 0,1].$
\end{definition}

\begin{lemma}
A function $F:\Delta \rightarrow
\mathbb{R}
_{\mathcal{I}}^{+}$ is interval-valued convex on co-ordinates if and only if
there exists two functions $F_{x}:[c,d]\rightarrow
\mathbb{R}
_{\mathcal{I}}^{+},~F_{x}(w)=F(x,w)$ and $F_{y}:[a,b]\rightarrow
\mathbb{R}
_{\mathcal{I}}^{+},~F_{y}(u)=F(y,u)$ are interval-valued convex.
\end{lemma}

The proof of this lemma follows immediately by the definition of
interval-valued co-ordinated convex function.

It is easy to proof that an inteval-valued convex function is inteval-valued
co-ordinated convex but the converse may not be true. For this we can see
the following example.

\begin{example}
An interval-valued function $F:[0,1]^{2}\rightarrow
\mathbb{R}
_{\mathcal{I}}^{+}$ defined as $F(x,y)=[xy,(6-e^{x})(6-e^{y})]$ is
interval-valued convex on co-ordinates but it is not interval-valued convex
on $[0,1]^{2}.$
\end{example}

\begin{proposition}
If $F,G:\Delta \rightarrow
\mathbb{R}
_{\mathcal{I}}^{+}$ are two interval-valued co-ordinated convex functions on
$\Delta $ and $\alpha \geq 0$, then $F+G$ and $\alpha F$ are interval-valued
co-ordinated convex functions.
\end{proposition}

\begin{proposition}
If $F,G:\Delta \rightarrow
\mathbb{R}
_{\mathcal{I}}^{+}$ are two interval-valued co-ordinated convex functions on
$\Delta ,$ then ($FG$) is interval-valued co-ordinated convex function on $%
\Delta .$
\end{proposition}

\begin{proof}
Since $F$ and $G$ are interval-valued co-ordinated convex functions, we have
\begin{eqnarray}
&&F(tx+(1-t)y,su+(1-s)w)  \label{c1} \\
&\supseteq &tsF(x,u)+t(1-s)F(x,w)+s(1-t)F(y,u)+(1-s)(1-t)F(y,w)  \notag
\end{eqnarray}

and
\begin{eqnarray}
&&G(tx+(1-t)y,su+(1-s)w)  \label{c2} \\
&\supseteq &tsG(x,u)+t(1-s)G(x,w)+s(1-t)G(y,u)+(1-s)(1-t)G(y,w).  \notag
\end{eqnarray}

Multiplying (\ref{c1}) and (\ref{c2}), we have%
\begin{eqnarray*}
&&F(tx+(1-t)y,su+(1-s)w)G(tx+(1-t)y,su+(1-s)w) \\
&\supseteq &\left[ tsF(x,u)+t(1-s)F(x,w)+s(1-t)F(y,u)+(1-s)(1-t)F(y,w)\right]
\\
&&\times \left[ tsG(x,u)+t(1-s)G(x,w)+s(1-t)G(y,u)+(1-s)(1-t)G(y,w)\right] \\
&=&stF(x,u)G(x,u)+t(1-s)F(x,w)G(x,w) \\
&&+(1-t)sF(y,u)G(y,u)+(1-t)(1-s)F(y,w)G(y,w)
\end{eqnarray*}

and therefore $(FG)$ is inteval-valued co-ordinated convex function.
\end{proof}

In what follows, without causing confusion, we will delete notations of ($R$%
), ($IR$) and~($ID$). We start with the following Theorem.

\begin{theorem}
\label{mt1} If $F:\Delta \rightarrow
\mathbb{R}
_{\mathcal{I}}^{+}$ is interval-valued co-ordinated convex function on $%
\Delta $ such that $F(t)=\left[ \underline{F(t)},\overline{F(t)}\right] $,
then following inequalities holds:%
\begin{eqnarray}
F\left( \frac{a+b}{2},\frac{c+d}{2}\right)  &\supseteq &\frac{1}{2}\left[
\frac{1}{b-a}\int_{a}^{b}F\left( x,\frac{c+d}{2}\right) dx+\frac{1}{d-c}%
\int_{c}^{d}F\left( \frac{a+b}{2},y\right) dy\right]   \label{c3} \\
&\supseteq &\frac{1}{(b-a)(d-c)}\int_{a}^{b}\int_{c}^{d}F(x,y)dydx  \notag \\
&\supseteq &\frac{1}{4}\left[ \frac{1}{b-a}\int_{a}^{b}F(x,c)dx+\frac{1}{b-a}%
\int_{a}^{b}F(x,d)dx\right.   \notag \\
&&\left. +\frac{1}{d-c}\int_{c}^{d}F(a,y)dy+\frac{1}{d-c}\int_{c}^{d}F(b,y)dy%
\right]   \notag \\
&\supseteq &\frac{F(a,c)+F(a,d)+F(b,c)+F(b,d)}{4}.  \notag
\end{eqnarray}
\end{theorem}

\begin{proof}
Since $F$ is interval -valued co-ordinated convex function on co-ordinates $%
\Delta $, then $F_{x}:[c,d]\rightarrow
\mathbb{R}
_{\mathcal{I}}^{+},~F_{x}(y)=F(x,y)$ is interval-valued convex function on $%
[c,d]$ and for all $x\in \lbrack a,b].$ From inequality (\ref{k2}), we have%
\begin{equation*}
F_{x}\left( \frac{c+d}{2}\right) \supseteq \frac{1}{d-c}\int_{c}^{d}F_{x}%
\left( y\right) dy\supseteq \frac{F_{x}(c)+F_{x}(d)}{2},
\end{equation*}

that can be written as%
\begin{equation}
F\left( x,\frac{c+d}{2}\right) \supseteq \frac{1}{d-c}\int_{c}^{d}F\left(
x,y\right) dy\supseteq \frac{F(x,c)+F(x,d)}{2}.  \label{c4}
\end{equation}

Integrating (\ref{c4}) with respect to $x$ over $[a,b]$ and dividing both
sides by $(b-a),$ we have%
\begin{eqnarray}
\frac{1}{b-a}\int_{a}^{b}F\left( x,\frac{c+d}{2}\right) dx &\supseteq &\frac{%
1}{(b-a)(d-c)}\int_{a}^{b}\int_{c}^{d}F\left( x,y\right) dydx  \label{c5} \\
&\supseteq &\frac{1}{2(b-a)}\left[ \int_{a}^{b}F(x,c)dx+\int_{a}^{b}F(x,d)dx%
\right] .  \notag
\end{eqnarray}

Similarly, $F_{y}=[a,b]\rightarrow
\mathbb{R}
_{\mathcal{I}}^{+},$ $F_{y}(x)=F(x,y)$ is interval-valued convex function on
$[a,b]$ and $y\in \lbrack c,d]$, we have
\begin{eqnarray}
\frac{1}{d-c}\int_{c}^{d}F\left( \frac{a+b}{2},y\right) dy &\supseteq &\frac{%
1}{(b-a)(d-c)}\int_{a}^{b}\int_{a}^{b}F(x,y)dydx  \label{c6} \\
&\supseteq &\frac{1}{2(d-c)}\left[ \int_{c}^{d}F(a,y)dy+\int_{c}^{d}F(b,y)dy%
\right] .  \notag
\end{eqnarray}

By adding (\ref{c5}) and (\ref{c6}) and using Theorem \ref{t1}, we have
second and third inequality in (\ref{c3}).

We also have from (\ref{k2}),%
\begin{eqnarray}
F\left( \frac{a+b}{2},\frac{c+d}{2}\right) &\supseteq &\frac{1}{b-a}%
\int_{a}^{b}F\left( x,\frac{c+d}{2}\right) dx  \label{c7} \\
F\left( \frac{a+b}{2},\frac{c+d}{2}\right) &\supseteq &\frac{1}{d-c}%
\int_{c}^{d}F\left( \frac{a+b}{2},y\right) dy.  \label{c8}
\end{eqnarray}

By adding (\ref{c7}) and (\ref{c8}) and using Theorem (\ref{t1}), we have
first inequality in (\ref{c3}).

At the end, again from (\ref{k1}) and Theorem (\ref{t1}), we have%
\begin{eqnarray*}
\frac{1}{b-a}\int_{a}^{b}F(x,c)dx &\supseteq &\frac{F(a,c)+F(b,c)}{2}, \\
\frac{1}{b-a}\int_{a}^{b}F(x,d)dx &\supseteq &\frac{F(a,d)+F(b,d)}{2}, \\
\frac{1}{d-c}\int_{c}^{d}F(a,y)dy &\supseteq &\frac{F(a,c)+F(a,d)}{2}, \\
\frac{1}{d-c}\int_{c}^{d}F(b,y)dy &\supseteq &\frac{F(b,c)+F(b,d)}{2}
\end{eqnarray*}

and proof is completed.
\end{proof}

\begin{example}
Suppose that $[a,b]=[0,1]$ and $[c,d]=[1,2].$ Let $F:[a,b]\times \lbrack
c,d]\rightarrow
\mathbb{R}
_{\mathcal{I}}^{+}$ be given as $F(x,y)=[xy,4xy]$, for all $x\in \lbrack
a,b] $ and $y\in \lbrack c,d].$ We have%
\begin{equation*}
F\left( \frac{a+b}{2},\frac{c+d}{2}\right) =\left[ \frac{3}{4},3\right] ,
\end{equation*}%
\begin{equation*}
\frac{1}{2}\left[ \frac{1}{b-a}\int_{a}^{b}F\left( x,\frac{c+d}{2}\right) dx+%
\frac{1}{d-c}\int_{c}^{d}F\left( \frac{a+b}{2},y\right) dy\right] =\left[
\frac{3}{4},3\right] ,
\end{equation*}%
\begin{equation*}
\frac{1}{(b-a)(d-c)}\int_{a}^{b}\int_{c}^{d}F(x,y)dydx=\left[ \frac{3}{4},3%
\right] ,
\end{equation*}%
\begin{equation*}
\frac{1}{4}\left[ \frac{1}{b-a}\int_{a}^{b}F(x,c)dx+\frac{1}{b-a}%
\int_{a}^{b}F(x,d)dx\right.
\end{equation*}%
\begin{equation*}
\left. +\frac{1}{d-c}\int_{c}^{d}F(a,y)dy+\frac{1}{d-c}\int_{c}^{d}F(b,y)dy%
\right] =\left[ \frac{3}{4},3\right] ,
\end{equation*}%
\begin{equation*}
\frac{F(a,c)+F(a,d)+F(b,c)+F(b,d)}{4}=\left[ \frac{3}{4},3\right] .
\end{equation*}

Hence $\left[ \frac{3}{4},3\right] \supseteq \left[ \frac{3}{4},3\right]
\supseteq \left[ \frac{3}{4},3\right] \supseteq \left[ \frac{3}{4},3\right]
. $
\end{example}

\begin{remark}
If $\overline{F}=\underline{F}$ in Theorem \ref{mt1}, then Theorem \ref{mt1}
reduces to Theorem \ref{Dr1}.
\end{remark}

\begin{theorem}
\label{mt2} If $F,G:\Delta \rightarrow
\mathbb{R}
_{\mathcal{I}}^{+}$ are two interval--valued co-ordinated convex functions
such that $F(t)=\left[ \underline{F(t)},\overline{F(t)}\right] $ and $G(t)=%
\left[ \underline{G(t)},\overline{G(t)}\right] ,$ then the following
inequality holds:%
\begin{equation}
\frac{1}{(b-a)(d-c)}\int_{a}^{b}\int_{c}^{d}F(x,y)G(x,y)dydx\supseteq \frac{1%
}{9}P(a,b,c,d)+\frac{1}{18}M(a,b,c,d)+\frac{1}{36}N(a,b,c,d),  \label{c9}
\end{equation}

where%
\begin{eqnarray*}
P(a,b,c,d) &=&F(a,c)G(a,c)+F(a,d)G(a,d)+F(b,d)G(b,d), \\
M(a,b,c,d) &=&F(a,c)G(a,d)+F(a,d)G(a,c)+F(b,c)G(b,d)+F(b,d)G(b,c) \\
&&+F(b,c)G(a,c)+F(a,c)G(b,c)+F(b,d)G(a,d)+F(a,d)G(b,d), \\
N(a,b,c,d) &=&F(b,c)G(a,d)+F(a,d)G(b,c)+F(b,d)G(a,c)+F(a,c)G(b,d).
\end{eqnarray*}
\end{theorem}

\begin{proof}
Since $F$ and $G$ are interval-valued co-ordinated convex functions on $%
\Delta $, therefore $F_{x}(y):[c,d]\rightarrow
\mathbb{R}
_{\mathcal{I}}^{+},~F_{x}(y)=F(x,y)$, $G_{x}(y):[c,d]\rightarrow
\mathbb{R}
_{\mathcal{I}}^{+},~G_{x}(y)=G(x,y)$ and $F_{y}(x):[a,b]\rightarrow
\mathbb{R}
_{\mathcal{I}}^{+},~F_{y}(x)=F(x,y),~G_{y}:[a,b]\rightarrow
\mathbb{R}
_{\mathcal{I}}^{+},~G_{y}(x)=G(x,y)$ are interval-valued convex functions on
$[c,d]$ and $[a,b]$ respectively for all $x\in \lbrack a,b],~y\in \lbrack
c,d].$

Now from inequality (\ref{k5}), we have%
\begin{eqnarray*}
\frac{1}{d-c}\int_{c}^{d}F_{x}(y)G_{x}(y)dy &\supseteq &\frac{1}{3}%
[F_{x}(c)G_{x}(c)+F_{x}(d)G_{x}(d)] \\
&&+\frac{1}{6}[F_{x}(c)G_{x}(d)+F_{x}(d)G_{x}(c)]
\end{eqnarray*}

that can be written as%
\begin{eqnarray*}
\frac{1}{d-c}\int_{c}^{d}F(x,y)G(x,y)dy &\supseteq &\frac{1}{3}%
[F(x,c)G(x,c)+F(x,d)G(x,d)] \\
&&+\frac{1}{6}[F(x,c)G(x,d)+F(x,d)G(x,c)].
\end{eqnarray*}

Integrating the above inequality with respect to $x$ over $[a,b]$ and and
dividing both sides by $b-a$, we have%
\begin{eqnarray}
&&\frac{1}{(b-a)(d-c)}\int_{a}^{b}\int_{c}^{d}F(x,y)G(x,y)dydx  \label{c10}
\\
&\supseteq &\frac{1}{3(b-a)}\int_{a}^{b}[F(x,c)G(x,c)+F(x,d)G(x,d)]dx  \notag
\\
&&+\frac{1}{6(b-a)}\int_{a}^{b}[F(x,c)G(x,d)+F(x,d)G(x,c)]dx.  \notag
\end{eqnarray}

Now using inequality (\ref{k5}) to each integral on right hand side of (\ref%
{c10}), we have%
\begin{eqnarray}
\frac{1}{b-a}\int_{a}^{b}F(x,c)G(x,c)dx &\supseteq &\frac{1}{3}%
[F(a,c)G(a,c)+F(b,c)G(b,c)]  \label{c11} \\
&&+\frac{1}{6}[F(a,c)G(b,c)+F(b,c)G(a,c)],  \notag \\
\frac{1}{b-a}\int_{a}^{b}F(x,d)G(x,d)dx &\supseteq &\frac{1}{3}%
[F(a,d)G(a,d)+F(b,d)G(b,d)]  \label{c12} \\
&&+\frac{1}{6}[F(a,d)G(b,d)+F(b,d)G(a,d)],  \notag \\
\frac{1}{b-a}\int_{a}^{b}F(x,c)G(x,d)dx &\supseteq &\frac{1}{3}%
[F(a,c)G(a,d)+F(b,c)G(b,d)]  \label{c13} \\
&&+\frac{1}{6}[F(a,c)G(b,d)+F(b,c)G(a,d)],  \notag \\
\frac{1}{b-a}\int_{a}^{b}F(x,d)G(x,c)dx &\supseteq &\frac{1}{3}%
[F(a,d)G(a,c)+F(b,d)G(b,c)]  \label{c14} \\
&&+\frac{1}{6}[F(a,d)G(b,c)+F(b,d)G(a,c)].  \notag
\end{eqnarray}

Substituting (\ref{c11})-(\ref{c14}) in (\ref{c10}), we have our desired
inequality (\ref{c9}). Similarly we can find same inequality by using $%
F_{y}(x)G_{y}(x)$ on $[a,b].$
\end{proof}

\begin{remark}
If $\overline{F}=\underline{F}$ in Theorem \ref{mt2}, then Theorem \ref{mt2}
reduces to (\cite[Theorem 4]{AL}).
\end{remark}

\begin{theorem}
\label{mt3}If $F,G:\Delta \rightarrow
\mathbb{R}
_{\mathcal{I}}^{+}$ are two interval--valued co-ordinated convex functions
such that $F(t)=\left[ \underline{F(t)},\overline{F(t)}\right] $ and $G(t)=%
\left[ \underline{G(t)},\overline{G(t)}\right] ,$ then we have the following
inequality:%
\begin{eqnarray}
&&4F\left( \frac{a+b}{2},\frac{c+d}{2}\right) G\left( \frac{a+b}{2},\frac{c+d%
}{2}\right)   \label{c15} \\
&\supseteq &\frac{1}{(b-a)(d-c)}\int_{a}^{b}\int_{c}^{d}F(x,y)G(x,y)dydx
\notag \\
&&+\frac{5}{36}P(a,b,c,d)+\frac{7}{36}M(a,b,c,d)+\frac{2}{9}N(a,b,c,d),
\notag
\end{eqnarray}

where $P(a,b,c,d)$, $M(a,b,c,d)$ and $N(a,b,c,d)$ are defined in Theorem \ref%
{mt2}.
\end{theorem}

\begin{proof}
Since $F~$and $G$ are interval-valued co-ordinated convex functions, from (%
\ref{k6}) we have
\begin{eqnarray}
&&2F\left( \frac{a+b}{2}\right) G\left( \frac{a+b}{2}\right)  \label{c16} \\
&\supseteq &\frac{1}{b-a}\int_{a}^{b}F\left( x,\frac{c+d}{2}\right) G\left(
x,\frac{c+d}{2}\right) dx  \notag \\
&&+\frac{1}{6}\left[ F\left( a,\frac{c+d}{2}\right) G\left( a,\frac{c+d}{2}%
\right) +F\left( b,\frac{c+d}{2}\right) G\left( b,\frac{c+d}{2}\right) %
\right]  \notag \\
&&+\frac{1}{3}\left[ F\left( a,\frac{c+d}{2}\right) G\left( b,\frac{c+d}{2}%
\right) +F\left( b,\frac{c+d}{2}\right) G\left( a,\frac{c+d}{2}\right) %
\right]  \notag
\end{eqnarray}

and
\begin{eqnarray}
&&2F\left( \frac{a+b}{2}\right) G\left( \frac{a+b}{2}\right)  \label{c17} \\
&\supseteq &\frac{1}{d-c}\int_{c}^{d}F\left( \frac{a+b}{2},y\right) G\left(
\frac{a+b}{2},y\right) dy  \notag \\
&&+\frac{1}{6}\left[ F\left( \frac{a+b}{2},c\right) G\left( \frac{a+b}{2}%
,c\right) +F\left( \frac{a+b}{2},d\right) G\left( \frac{a+b}{2},d\right) %
\right]  \notag \\
&&+\frac{1}{3}\left[ F\left( \frac{a+b}{2},c\right) G\left( \frac{a+b}{2}%
,d\right) +F\left( \frac{a+b}{2},d\right) G\left( \frac{a+b}{2},c\right) %
\right] .  \notag
\end{eqnarray}

Adding (\ref{c16}), (\ref{c18}) and multiplying on both sides of the
resultnat one by $2,$ we get%
\begin{eqnarray}
&&8\left( \frac{a+b}{2}\right) G\left( \frac{a+b}{2}\right)  \label{c18} \\
&\supseteq &\frac{2}{b-a}\int_{a}^{b}F\left( x,\frac{c+d}{2}\right) G\left(
x,\frac{c+d}{2}\right) dx  \notag \\
&&+\frac{2}{d-c}\int_{c}^{d}F\left( \frac{a+b}{2},y\right) G\left( \frac{a+b%
}{2},y\right) dy  \notag \\
&&+\frac{1}{6}\left[ 2F\left( a,\frac{c+d}{2}\right) G\left( a,\frac{c+d}{2}%
\right) +2F\left( b,\frac{c+d}{2}\right) G\left( b,\frac{c+d}{2}\right) %
\right]  \notag \\
&&+\frac{1}{6}\left[ 2F\left( \frac{a+b}{2},c\right) G\left( \frac{a+b}{2}%
,c\right) +2F\left( \frac{a+b}{2},d\right) G\left( \frac{a+b}{2},d\right) %
\right]  \notag \\
&&+\frac{1}{3}\left[ F\left( a,\frac{c+d}{2}\right) G\left( b,\frac{c+d}{2}%
\right) +F\left( b,\frac{c+d}{2}\right) G\left( a,\frac{c+d}{2}\right) %
\right]  \notag \\
&&+\frac{1}{3}\left[ F\left( \frac{a+b}{2},c\right) G\left( \frac{a+b}{2}%
,d\right) +F\left( \frac{a+b}{2},d\right) G\left( \frac{a+b}{2},c\right) %
\right] .  \notag
\end{eqnarray}

Now from (\ref{k6}), we have
\begin{eqnarray}
&&2F\left( a,\frac{c+d}{2}\right) G\left( a,\frac{c+d}{2}\right)  \label{c19}
\\
&\supseteq &\frac{1}{d-c}\int_{c}^{d}F(a,y)G(a,y)dy  \notag \\
&&+\frac{1}{6}\left[ F(a,c)G(a,c)+F(a,d)G(a,d)\right]  \notag \\
&&+\frac{1}{3}[F(a,c)G(a,d)+F(a,d)G(a,c)],  \notag
\end{eqnarray}%
\begin{eqnarray}
&&2F\left( b,\frac{c+d}{2}\right) G\left( b,\frac{c+d}{2}\right)  \label{c20}
\\
&\supseteq &\frac{1}{d-c}\int_{c}^{d}F(b,y)G(b,y)dy  \notag \\
&&+\frac{1}{6}[F(b,c)G(b,c)+F(b,d)G(b,d)]  \notag \\
&&+\frac{1}{3}[F(b,c)G(b,d)+F(b,d)G(b,c)],  \notag
\end{eqnarray}%
\begin{eqnarray}
&&2F\left( \frac{a+b}{2},c\right) G\left( \frac{a+b}{2},c\right)  \label{c21}
\\
&\supseteq &\frac{1}{b-a}\int_{a}^{b}F(x,c)G(x,c)dx  \notag \\
&&+\frac{1}{6}[F(a,c)G(a,c)+F(b,c)G(b,c)]  \notag \\
&&+\frac{1}{3}[F(a,c)G(b,c)+F(b,c)G(a,c)],  \notag
\end{eqnarray}%
\begin{eqnarray}
&&2F\left( \frac{a+b}{2},d\right) G\left( \frac{a+b}{2},d\right)  \label{c23}
\\
&\supseteq &\frac{1}{b-a}\int_{a}^{b}F(x,d)G(x,d)dx  \notag \\
&&+\frac{1}{6}[F(a,d)G(a,d)+F(b,d)G(b,d)]  \notag \\
&&+\frac{1}{3}[F(a,d)G(b,d)+F(b,d)G(a,d)],  \notag
\end{eqnarray}%
\begin{eqnarray}
&&2F\left( a,\frac{c+d}{2}\right) G\left( b,\frac{c+d}{2}\right)  \label{c24}
\\
&\supseteq &\frac{1}{d-c}\int_{c}^{d}F(a,y)G(b,y)dy  \notag \\
&&+\frac{1}{6}[F(a,c)G(b,c)+F(a,d)G(b,d)]  \notag \\
&&+\frac{1}{3}[F(a,c)G(b,d)+F(a,d)G(b,c)],  \notag
\end{eqnarray}%
\begin{eqnarray}
&&2F\left( b,\frac{c+d}{2}\right) G\left( a,\frac{c+d}{2}\right)  \label{c25}
\\
&\supseteq &\frac{1}{d-c}\int_{c}^{d}F(b,y)G(a,y)dy  \notag \\
&&+\frac{1}{6}[F(b,c)G(a,c)+F(b,d)G(a,d)]  \notag \\
&&+\frac{1}{3}[F(b,c)G(a,d)+F(b,d)G(a,c)],  \notag
\end{eqnarray}%
\begin{eqnarray}
&&2F\left( \frac{a+b}{2},c\right) G\left( \frac{a+b}{2},d\right)  \label{c26}
\\
&\supseteq &\frac{1}{b-a}\int_{a}^{b}F(x,c)G(x,d)dx  \notag \\
&&+\frac{1}{6}[F(a,c)G(a,d)+F(b,c)G(b,d)]  \notag \\
&&+\frac{1}{3}[F(a,c)G(b,d)+F(b,c)G(a,d)],  \notag
\end{eqnarray}%
\begin{eqnarray}
&&2F\left( \frac{a+b}{2},d\right) G\left( \frac{a+b}{2},c\right)  \label{c27}
\\
&\supseteq &\frac{1}{b-a}\int_{a}^{b}F(x,d)G(x,c)dx  \notag \\
&&+\frac{1}{6}[F(a,d)G(a,c)+F(b,d)G(b,c)]  \notag \\
&&+\frac{1}{3}[F(a,d)G(b,c)+F(b,d)G(a,c)].  \notag
\end{eqnarray}

Using (\ref{c19})-(\ref{c27}) in (\ref{c18}), we have%
\begin{eqnarray}
&&8\left( \frac{a+b}{2}\right) G\left( \frac{a+b}{2}\right)  \label{c28} \\
&\supseteq &\frac{2}{b-a}\int_{a}^{b}F\left( x,\frac{c+d}{2}\right) G\left(
x,\frac{c+d}{2}\right) dx  \notag \\
&&+\frac{2}{d-c}\int_{c}^{d}F\left( \frac{a+b}{2},y\right) G\left( \frac{a+b%
}{2},y\right) dy  \notag \\
&&+\frac{1}{6(d-c)}\int_{c}^{d}F(a,y)G(a,y)dy+\frac{1}{6(d-c)}%
\int_{c}^{d}F(b,y)G(b,y)  \notag \\
&&+\frac{1}{6(b-a)}\int_{a}^{b}F(x,c)G(x,c)dx+\frac{1}{6(b-a)}%
\int_{a}^{b}F(x,d)G(x,d)dx  \notag \\
&&+\frac{1}{3(d-c)}\int_{c}^{d}F(a,y)G(b,y)dy+\frac{1}{3(d-c)}%
\int_{c}^{d}F(b,y)G(a,y)dy  \notag \\
&&+\frac{1}{3(b-a)}\int_{a}^{b}F(x,c)G(x,d)dx+\frac{1}{3(b-a)}%
\int_{a}^{b}F(x,d)G(x,c)dx  \notag \\
&&+\frac{1}{18}P(a,b,c,d)+\frac{1}{9}M(a,b,c,d)+\frac{2}{9}N(a,b,c,d).
\notag
\end{eqnarray}

Again from (\ref{k6}), we have following relations%
\begin{eqnarray}
&&\frac{2}{d-c}\int_{c}^{d}F\left( \frac{a+b}{2},y\right) G\left( \frac{a+b}{%
2},y\right) dy  \label{c29} \\
&\supseteq &\frac{1}{(b-a)(d-c)}\int_{a}^{b}\int_{c}^{d}F(x,y)G(x,y)dydx
\notag \\
&&+\frac{1}{6(d-c)}\int_{c}^{d}F(a,y)G(a,y)dy+\frac{1}{6(d-c)}%
\int_{c}^{d}F(b,y)G(b,y)dy  \notag \\
&&+\frac{1}{3(d-c)}\int_{c}^{d}F(a,y)G(b,y)dy+\frac{1}{3(d-c)}%
\int_{c}^{d}F(b,y)G(a,y)dy,  \notag
\end{eqnarray}%
\begin{eqnarray}
&&\frac{2}{b-a}\int_{a}^{b}F\left( x,\frac{c+d}{2}\right) G\left( x,\frac{c+d%
}{2}\right) dx  \label{c30} \\
&\supseteq &\frac{1}{(b-a)(d-c)}\int_{a}^{b}\int_{c}^{d}F(x,y)G(x,y)  \notag
\\
&&+\frac{1}{6(b-a)}\int_{a}^{b}F(x,c)G(x,c)dx+\frac{1}{6(b-a)}%
\int_{a}^{b}F(x,d)G(x,d)dx  \notag \\
&&+\frac{1}{3(b-a)}\int_{a}^{b}F(x,c)G(x,d)dx+\frac{1}{3(b-a)}%
\int_{a}^{b}F(x,d)G(x,c)dx.  \notag
\end{eqnarray}

Using (\ref{c29}) and (\ref{c30}) in (\ref{c28}), we have%
\begin{eqnarray}
&&8\left( \frac{a+b}{2}\right) G\left( \frac{a+b}{2}\right)  \label{c31} \\
&\supseteq &\frac{2}{(b-a)(d-c)}\int_{a}^{b}\int_{c}^{d}F(x,y)G(x,y)dydx
\notag \\
&&+\frac{1}{3(d-c)}\int_{c}^{d}F(a,y)G(a,y)dy+\frac{1}{3(d-c)}%
\int_{c}^{d}F(b,y)G(b,y)dy  \notag \\
&&+\frac{1}{3(b-a)}\int_{a}^{b}F(x,c)G(x,c)dx+\frac{1}{3(b-a)}%
\int_{a}^{b}F(x,d)G(x,d)dx  \notag \\
&&+\frac{2}{3(d-c)}\int_{c}^{d}F(a,y)G(b,y)dy+\frac{2}{3(d-c)}%
\int_{c}^{d}F(b,y)G(a,y)dy  \notag \\
&&+\frac{2}{3(b-a)}\int_{a}^{b}F(x,c)G(x,d)dx+\frac{2}{3(b-a)}%
\int_{a}^{b}F(x,d)G(x,c)dx  \notag \\
&&+\frac{1}{18}P(a,b,c,d)+\frac{1}{9}M(a,b,c,d)+\frac{2}{9}N(a,b,c,d)  \notag
\end{eqnarray}

and by using (\ref{k5}) on each integral in (\ref{c31}), we have our
required result.
\end{proof}

\begin{remark}
If $\overline{F}=\underline{F}$ in Theorem \ref{mt3}, then Theorem \ref{mt3}
reduces to the (\cite[Theorem 5]{AL}).
\end{remark}

\section*{Conclusion}

In this article, interval-valued co-ordinated convex function and double
integral for the interval-valued functions are introduced and establish some
new inequlities of Hermite-Hadamard type. Our inequalilities are the extend
some previously obtained results.

\bibliographystyle{Plain}
\bibliography{Dafang_2019_12_27}

\begin{thebibliography}{10}

\bibitem{Almori}
Mohammad Alomari and Maslina Darus.
\newblock Co-ordinated s-convex function in the first sense with some
  hadamard-type inequalities.
\newblock {\em Int. J. Contemp. Math. Sci}, 3(32):1557--1567, 2008.

\bibitem{Breckner}
Wolfgang~W Breckner.
\newblock Continuity of generalized convex and generalized concave set-valued
  functions.
\newblock {\em Rev. Anal. Num{\'e}r. Th{\'e}or. Approx.}, 22(1):39--51, 1993.

\bibitem{CAno}
Yurilev Chalco-Cano, Arturo Flores-Franulic, and Heriberto Rom{\'a}n-Flores.
\newblock Ostrowski type inequalities for interval-valued functions using
  generalized hukuhara derivative.
\newblock {\em Computational \& Applied Mathematics}, 31(3), 2012.

\bibitem{Cano1}
Yurilev Chalco-Cano, Weldon~A Lodwick, and W~Condori-Equice.
\newblock Ostrowski type inequalities and applications in numerical integration
  for interval-valued functions.
\newblock {\em Soft Computing}, 19(11):3293--3300, 2015.

\bibitem{costa}
TM~Costa.
\newblock Jensen's inequality type integral for fuzzy-interval-valued
  functions.
\newblock {\em Fuzzy Sets and Systems}, 327:31--47, 2017.

\bibitem{costa2}
TM~Costa and Heriberto Rom{\'a}n-Flores.
\newblock Some integral inequalities for fuzzy-interval-valued functions.
\newblock {\em Information Sciences}, 420:110--125, 2017.

\bibitem{Dragomir}
SS~Dragomir.
\newblock On the hadamard's inequlality for convex functions on the
  co-ordinates in a rectangle from the plane.
\newblock {\em Taiwanese Journal of Mathematics}, pages 775--788, 2001.

\bibitem{Dragomir1}
SS~Dragomir and CEM Pearce.
\newblock Selected topics on hermite-hadamard inequalities and applications,
  rgmia monographs, victoria university, 2000.
\newblock {\em ONLINE: http://rgmia. vu. edu. au/monographs}, 2004.

\bibitem{flo}
Arturo Flores-Franuli{\v{c}}, Yurilev Chalco-Cano, and Heriberto
  Rom{\'a}n-Flores.
\newblock An ostrowski type inequality for interval-valued functions.
\newblock In {\em 2013 Joint IFSA World Congress and NAFIPS Annual Meeting
  (IFSA/NAFIPS)}, pages 1459--1462. IEEE, 2013.

\bibitem{AL}
MA~Latif and M~Alomari.
\newblock Hadamard-type inequalities for product two convex functions on the
  co-ordinates.
\newblock In {\em Int. Math. Forum}, volume~4, pages 2327--2338, 2009.

\bibitem{Mitroi}
Flavia-Corina Mitroi, Kazimierz Nikodem, and Szymon Wasowicz.
\newblock Hermite--hadamard inequalities for convex set-valued functions.
\newblock {\em Demonstratio Mathematica}, 46(4):655--662, 2013.

\bibitem{Moore}
Ramon~E Moore.
\newblock {\em Interval analysis}.
\newblock Prentice-Hall, Englewood Cliffs, 1966.

\bibitem{Moore2}
Ramon~E Moore, R~Baker Kearfott, and Michael~J Cloud.
\newblock {\em Introduction to interval analysis}.
\newblock Siam, Philadelphia, P. A., 2009.

\bibitem{niko1}
Kazimierz Nikodem, Jose~Luis Sanchez, and Luisa Sanchez.
\newblock Jensen and hermite-hadamard inequalities for strongly convex
  set-valued maps.
\newblock {\em Mathematica Aeterna}, 4(8):979--987, 2014.

\bibitem{Gomez}
Rafaela Osuna-Gomez, Maria~Dolores Jimenez-Gamero, Yurilev Chalco-Cano, and
  Marko~Antonio Rojas-Medar.
\newblock Hadamard and jensen inequalities for s-convex fuzzy processes.
\newblock In {\em Soft methodology and random information systems}, pages
  645--652. Springer, Berlin/Heidelberg, 2004.

\bibitem{MZ}
M~Emin Ozdemir, Erhan Set, and Mehmet~Zeki Sar{\i}kaya.
\newblock Some new hadamard type inequalities for co-ordinated.
\newblock {\em Hacettepe Journal of Mathematics and Statistics},
  40(2):219--229, 2011.

\bibitem{Pecaric}
Josip~E Peajcariaac and Yung~Liang Tong.
\newblock {\em Convex functions, partial orderings, and statistical
  applications}.
\newblock Academic Press, Bostan SanDiego New York London Sydney Tokyo Toronto,
  1992.

\bibitem{flo2}
H~Rom{\'a}n-Flores, Y~Chalco-Cano, and WA~Lodwick.
\newblock Some integral inequalities for interval-valued functions.
\newblock {\em Computational and Applied Mathematics}, 37(2):1306--1318, 2018.

\bibitem{flo3}
Heriberto Rom{\'a}n-Flores, Yurilev Chalco-Cano, and Geraldo~Nunes Silva.
\newblock A note on gronwall type inequality for interval-valued functions.
\newblock In {\em 2013 Joint IFSA World Congress and NAFIPS Annual Meeting
  (IFSA/NAFIPS)}, pages 1455--1458. IEEE, 2013.

\bibitem{Sado}
Elzbieta Sadowska.
\newblock Hadamard inequality and a refinement of jensen inequality for set
  valued functions.
\newblock {\em Results in Mathematics}, 32(3-4):332--337, 1997.

\bibitem{Zhao}
Dafang Zhao, Tianqing An, Guoju Ye, and Wei Liu.
\newblock New jensen and hermite--hadamard type inequalities for h-convex
  interval-valued functions.
\newblock {\em Journal of Inequalities and Applications}, 2018(1):302, 2018.

\bibitem{zhao3}
Dafang Zhao, Tianqing An, Guoju Ye, and Wei Liu.
\newblock Chebyshev type inequalities for interval-valued functions.
\newblock {\em Fuzzy Sets and Systems}, 2019.

\bibitem{zhao2}
Dafang Zhao, Guoju Ye, Wei Liu, and Delfim~FM Torres.
\newblock Some inequalities for interval-valued functions on time scales.
\newblock {\em Soft Computing}, 23(15):6005--6015, 2019.

\end{thebibliography}

\end{document}